\documentclass[11pt,a4paper]{amsart}

\usepackage{amsmath,amssymb,amsthm}
\usepackage[margin=1.1in]{geometry}
\usepackage{graphicx}
\usepackage{xcolor}
\usepackage{hyperref}

\hypersetup{colorlinks=true,linkcolor=blue!50!black,citecolor=blue!50!black,
            urlcolor=blue!50!black,
            pdftitle={The `1/9'-problem for the phi-functions},
            pdfauthor={Thomas Schmelzer}}

\newtheorem{theorem}{Theorem}[section]
\newtheorem{lemma}[theorem]{Lemma}
\newtheorem{corollary}[theorem]{Corollary}
\theoremstyle{definition}
\newtheorem{remark}[theorem]{Remark}

\newcommand{\Rm}{\mathbb{R}_-}
\newcommand{\C}{\mathbb{C}}
\newcommand{\rstar}[1]{r^{*}_{#1}}

\begin{document}

\title[The `1/9'-problem for the $\varphi$-functions]
      {The `1/9'-problem for the $\varphi$-functions}

\author{Thomas Schmelzer}
\address{Jebel Quant Research, Abu Dhabi, United Arab Emirates}
\email{thomas@jqr.ae}

\subjclass[2020]{Primary 41A20, 41A25; Secondary 30E10, 65F60, 65L04}

\keywords{Rational approximation, `1/9'-problem, Halphen's constant,
$\varphi$-functions, exponential integrators, equilibrium distributions}

\thanks{All computations are reproducible from the scripts accompanying this
paper.}

\date{13 September 2026}

\dedicatory{In memory of Herbert Stahl (3 August 1942 -- 22 April 2013)}

\begin{abstract}
The functions $\varphi_\ell(z)=(e^z-s_\ell(z))/z^\ell$, with $s_\ell$ the
Taylor polynomial of $e^z$ of degree $\ell-1$, are what exponential
integrators evaluate, and they evaluate them through rational approximations
on the negative real axis $\mathbb{R}_-$. Schmelzer and Trefethen (2007/08) conjectured that the best such approximations
converge at the rate of the classical `1/9'-problem:
$E_n(\varphi_\ell)^{1/n}\to H$, Halphen's constant, for every $\ell$. We prove
it. The theorem behind it gives the same rate for $f=u_0+u_1\exp$ with $u_0$
rational and $u_1\not\equiv0$ meromorphic with finitely many poles and
$\log|u_1(z)|=o(|z|)$, the lower bound required only away from $\mathbb{R}_-$,
whenever $f$ has no pole on $\mathbb{R}_-$ (Theorem~4.4). For rational $u_1$ this is
Theorem 1 of Stahl and Schmelzer (2009), announced with its proof deferred to a manuscript that never
appeared; the hypotheses admit, for instance,
$u_1(z)=\sinh(\pi\sqrt z)/(\pi\sqrt z)$, whose zeros are infinite in number.

The proof is a reading of Gonchar and Rakhmanov (1989) rather than an extension of them. Their
theorem concerns a sequence given by a Cauchy-type integral with a varying
weight $\Phi_n$, and a contour identity of Schmelzer and Trefethen presents $u_1\exp$, once its
principal parts are removed, in exactly that form, with
$\Phi_n(t)=u_1(-nt)e^{-nt}$. The factor $u_1(-nt)$ grows subexponentially in
$n$, and their hypothesis on $\Phi_n$ divides by $2n$, so it drops out of the
external field: field, extremal arc, $S$-property and constant are all theirs,
unchanged. For $\varphi_\ell$ the factor is $t^{-\ell}$, whose pole sits at the
origin, a point of the approximation set and hence at positive distance from
every admissible contour.
\end{abstract}

\maketitle

\section{Introduction}

Exponential integrators for $\dot u=Au+g(u,t)$ are among the fastest methods
for stiff semilinear problems, and their implementation turns on evaluating
$f(A)b$ for $f$ in the family
\begin{equation}\label{eq:phi}
   \varphi_\ell(z)=\frac{e^z-s_\ell(z)}{z^\ell},
   \qquad s_\ell(z)=\sum_{k=0}^{\ell-1}\frac{z^k}{k!},
   \qquad \varphi_0=\exp .
\end{equation}
Each $\varphi_\ell$ is entire. When $A$ is negative semidefinite the relevant
spectrum lies on $\Rm=(-\infty,0]$, and a standard approach is to replace
$\varphi_\ell$ by a rational approximation on $\Rm$ and evaluate it in partial
fractions, so that only shifted linear systems $(A-z_jI)x=b$, one per pole
$z_j$, have to be solved
\cite{SchmelzerTrefethen,TWS}. Everything then depends on how well
$\varphi_\ell$ can be approximated by rational functions on $\Rm$, and on where
the poles of those approximations sit.

Write $\mathcal R_{n,m}$ for the rational functions of type $(n,m)$, that is
with numerator degree at most $n$ and denominator degree at most $m$, let
$\rstar{n,m}(f,\Rm;\cdot)$ be the best uniform approximant to $f$ on $\Rm$ from
$\mathcal R_{n,m}$, and put
\[
   E_n(f):=\bigl\|f-\rstar{n,n}(f,\Rm;\cdot)\bigr\|_{\Rm}.
\]
Let $H=0.1076539192\ldots=1/9.28902549\ldots$ be Halphen's constant. For
$f=\exp$ the classical `1/9'-problem gives $E_n(\exp)^{1/n}\to H$, and in the
sharp form of Magnus \cite{Magnus1994} and Aptekarev \cite{Aptekarev},
$E_n(\exp)=2H^{n+1/2}(1+o(1))$. For
$\ell\ge1$ much less is on record. The result of this note is that nothing
changes:
\[
   \lim_{n\to\infty}E_n(\varphi_\ell)^{1/n}=H\qquad(\ell=0,1,2,\ldots),
\]
Corollary~\ref{thm:conj} below; and more generally
$\lim_nE_n(u_0+u_1\exp)^{1/n}=H$ for $u_0$ rational and $u_1$ meromorphic with
finitely many poles and subexponential growth, Theorem~\ref{thm:class}.

Section~\ref{sec:history} sketches the history, which is also how I came to
the problem. Section~\ref{sec:rep} records the three facts about
$\varphi_\ell$ that are used. Section~\ref{sec:lower} proves
Theorem~\ref{thm:class} and, as Corollary~\ref{thm:conj}, Conjecture 2.1. The
proof has three steps and no new analysis: $u_0+u_1\exp$ is reduced, by
subtracting principal parts, to an entire $F$; $F(-nu)$ is written as a Cauchy
integral whose weight is $u_1(-nt)e^{-nt}$; and that weight differs from the
$e^{-nt}$ of Gonchar and Rakhmanov \cite{GoncharRakhmanov} by a factor their
hypothesis cannot see, because the factor grows subexponentially in $n$ and
the hypothesis divides by $2n$. Put generally, a subexponential factor in the
weight leaves the rate of their theorem unchanged; the theorem here is that
observation applied once. Section~\ref{sec:further} says what the proof
leaves open.

\section{A short history}\label{sec:history}

Cody, Meinardus and Varga \cite{CMV} showed in 1969 that $E_n(\exp)$ decays
geometrically, and Saff and Varga \cite{SaffVarga} put the rate on the list of
open questions. Sch\"onhage \cite{Schoenhage} proved that for the \emph{column}
approximants, numerator degree $0$, it is exactly $\tfrac13$; Meinardus and
Varga \cite{MeinardusVarga} extended the geometric decay from $\exp$ to
reciprocals of entire functions of perfectly regular growth with non-negative
coefficients, an early sign that the phenomenon belongs to a class and not to
one function. A diagonal approximant has twice as many free parameters, so
$\bigl(\tfrac13\bigr)^2=\tfrac19$ was the natural guess, and the problem
acquired the name it still carries.

The guess is wrong, and it was the Carath\'eodory--Fej\'er method that showed
it.
Developed by Gutknecht and Trefethen for real polynomial
\cite{GutknechtTrefethen} and then real rational \cite{TrefethenGutknecht}
approximation, CF produces near-best approximations cheaply and accurately
enough to read off the rate: Trefethen and Gutknecht \cite{TrefethenGutknecht}
found $0.1076539\ldots$, close to $1/9$ but not it, and the Remez computations
of Carpenter, Ruttan and Varga \cite{CRV} confirmed the digits. Magnus
\cite{Magnus1986} then \emph{conjectured} the closed form
\[
   H=\exp\Bigl(-\pi\frac{K'}{K}\Bigr),\qquad K(k)=2E(k),
\]
$K,K'$ complete elliptic integrals of the first kind for conjugate moduli and
$E$ that of the second, and returned to the constant over the following
fifteen years \cite{Magnus1994,MagnusMeinguet}.

The proof that $\lim E_n(\exp)^{1/n}=H$ is due to Gonchar and Rakhmanov
\cite{GoncharRakhmanov,Gonchar1986}. A best approximant interpolates $f$ at
the $2n+1$ zeros of its own error, so its denominator is orthogonal with
respect to a \emph{varying} complex weight, and the arc on which the poles
accumulate is not given in advance: it is singled out by a symmetry ($S$-)
property, and Herbert Stahl's theory \cite{Stahl1985,Stahl1986} of extremal
domains and of orthogonal polynomials with complex weights is what makes such
an arc well defined. Solving the resulting equilibrium problem for a condenser
in an external field, Gonchar and Rakhmanov obtained $H$ as the positive root
of $\sum_{l\ge1}a_lH^{\,l}=\tfrac18$, $a_l=\bigl|\sum_{b\mid l}(-1)^bb\bigr|$,
a constant Halphen had studied in an equivalent form in 1886 \cite{Halphen},
whence the name. Magnus's sharper $E_n(\exp)=2H^{n+1/2}(1+o(1))$
\cite{Magnus1994} was proved by Aptekarev \cite{Aptekarev} in 2002.

The $\varphi$-functions enter through \cite{TWS}, where Trefethen, Weideman
and I built CF approximations to $e^z$ on $\Rm$ as a competitor to Talbot
contours for inverting Laplace transforms, CF giving a type $(N,N)$
approximation indistinguishable from best by $N\approx9$. Exponential
integrators need the same for $\varphi_\ell$, and with that the `1/9'-problem
reappeared: is the rate for $\varphi_\ell$ again $H$? Numerically it plainly
was, but nothing in the literature covered a function other than $\exp$, and
\cite{SchmelzerTrefethen} could only record it as a conjecture. To settle it
Herbert Stahl and I took up the class $f=u_0+u_1\exp$ with $u_0,u_1$ rational,
which contains every $\varphi_\ell$; the result was \cite{StahlSchmelzer}. That
paper is an announcement: it sketches its proofs in \S4 and defers them to a
forthcoming paper, its reference [12], which never appeared. Its Theorems 1,
2, 4 and 7--10 therefore stand as announced results whose proofs are not in
print; nothing below rests on any of them.

\section{The $\varphi$-functions}\label{sec:rep}

Three facts about $\varphi_\ell$ are needed, and no more. The first two are
elementary and are illustrated in Figure~\ref{fig:phi}; the third, the contour
integral of \cite[Thm.~5.1]{SchmelzerTrefethen}, is the one the whole argument
turns on.

\begin{lemma}\label{lem:rep}
Let $\ell\ge1$. Then
\begin{enumerate}
\item $z^\ell\varphi_\ell(z)=e^z-s_\ell(z)$, so that
      $\varphi_\ell=u_0+u_1\exp$ with the rational functions
      $u_1(z)=z^{-\ell}$ and $u_0(z)=-s_\ell(z)/z^\ell$; and
      $\varphi_\ell(z)=\sum_{k\ge0}z^k/(k+\ell)!$ is entire.
\item $\varphi_\ell$ is positive and strictly increasing on the real axis, so
      on $\Rm$
      \[
         0<\varphi_\ell(x)\le\varphi_\ell(0)=\frac1{\ell!},
         \qquad
         \varphi_\ell(x)=\frac1{(\ell-1)!\,|x|}+O(|x|^{-2})
         \quad(x\to-\infty);
      \]
      in particular $\varphi_\ell$ extends continuously to $\Rm\cup\{\infty\}$
      with value $0$ at $\infty$, and $E_n(\varphi_\ell)$ is finite.
\item \emph{(Cauchy integral \cite[Thm.~5.1]{SchmelzerTrefethen}.)}
      For any contour $C$ winding once around $0$ and once around $z$,
      \begin{equation}\label{eq:cauchyphi}
         \varphi_\ell(z)=\frac1{2\pi i}\int_C\frac{e^s}{s^\ell}\,\frac{ds}{s-z} .
      \end{equation}
\end{enumerate}
\end{lemma}

\begin{proof}
(i) is \eqref{eq:phi} together with the cancellation of the first $\ell$ Taylor
coefficients of $e^z-s_\ell$. For (ii), integration by parts and the value at
$\ell=1$ identify $\varphi_\ell$ with the Duhamel integral
$\frac1{(\ell-1)!}\int_0^1e^{(1-\theta)z}\theta^{\ell-1}d\theta$; its integrand
is positive, and differentiating under the integral sign gives
$\varphi_\ell'>0$ on the real axis. The asymptotics follow from (i) with
$e^x\to0$ and $s_\ell(x)\sim x^{\ell-1}/(\ell-1)!$. For (iii), the integrand
has poles at $s=0$, of order $\ell$, and at $s=z$; the residue at $s=z$ is
$e^z/z^\ell$, and expanding $1/(s-z)=-\sum_{k\ge0}s^kz^{-k-1}$ shows the
residue at $s=0$ to be $-\sum_{k<\ell}z^{k-\ell}/k!=-s_\ell(z)/z^\ell$. Summing
the two and comparing with (i) gives \eqref{eq:cauchyphi}.
\end{proof}

\begin{figure}[htbp]
\centering
\includegraphics[width=\textwidth]{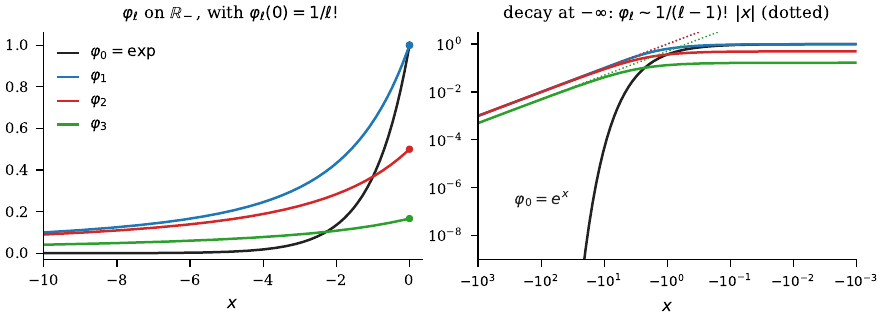}
\caption{Lemma~\ref{lem:rep}(ii). Left: $\varphi_\ell$ is positive and
increasing on $\Rm$, with maximum $\varphi_\ell(0)=1/\ell!$ (dots). Right: the
decay at $-\infty$ is only algebraic, $\varphi_\ell(x)\sim1/(\ell-1)!\,|x|$
(dotted), where $\varphi_0=\exp$ decays exponentially. Both are bounded on
$\Rm$ and vanish at $\infty$, which is what makes the approximation problem on
an unbounded set well posed.}
\label{fig:phi}
\end{figure}

So $\varphi_\ell$ belongs to the class $u_0+u_1\exp$ with $u_0,u_1$ rational
that Stahl and I took up in \cite{StahlSchmelzer}, and it belongs to
it in a slightly delicate way: $u_0$ and $u_1$ both have a pole of order $\ell$
at the origin, an endpoint of $\Rm$, and only their combination is bounded
there. Section~\ref{sec:lower} treats the class, and this is the feature it has
to accommodate.

\section{The rate for $u_0+u_1\exp$}\label{sec:lower}

The natural attempt is to convert an approximant to $\varphi_\ell$ into one to
$\exp$ and quote the classical rate \cite{GoncharRakhmanov}. It cannot be done,
and it is worth seeing why before looking for something better.

\begin{remark}[no rational transfer]\label{prop:notransfer}
Let $\ell\ge1$. An identity $h\varphi_\ell+c=u\exp+v$ with $h,c$ rational
forces $u=h/z^\ell$ and $v=c-hs_\ell/z^\ell$, and the transfer it suggests,
$R\mapsto(hR+c-v)/u$, is one and the same map whatever $h$ and $c$ are:
\[
   \frac{hR+c-v}{u}=\frac{hR+hs_\ell/z^\ell}{h/z^\ell}=z^\ell R+s_\ell .
\]
Its loss is exact, $\exp-(z^\ell R+s_\ell)=z^\ell(\varphi_\ell-R)$, and
unbounded on $\Rm$: the constant $R=1$ has $\|\varphi_\ell-1\|_{\Rm}\le1$
while $z^\ell(\varphi_\ell-1)=e^z-s_\ell-z^\ell$ is not bounded. So no $M$ has
$\|\exp-(z^\ell R+s_\ell)\|_{\Rm}\le M\|\varphi_\ell-R\|_{\Rm}$ for all
rational $R$, and no choice of $h$ can change that, since $h$ cancels. The
obstruction is the one that blocks every soft route: $z^\ell$ is unbounded at
one end of $\Rm$ and $z^{-\ell}$ at the other, and $\varphi_\ell$ is built from
$\exp$ by exactly that multiplication.
\end{remark}

The way through is not to transfer the approximant at all, but to transfer the
\emph{representation}. Gonchar and Rakhmanov \cite{GoncharRakhmanov} prove
their theorem not for a
fixed function but for a sequence given by a Cauchy-type integral with a
varying weight, and $\varphi_\ell$ has exactly such a representation --- it is
Lemma~\ref{lem:rep}(iii). We recall their statement in the form they use.

\begin{theorem}[{Gonchar--Rakhmanov \cite[Thm.~1]{GoncharRakhmanov}}]
\label{thm:GR}
Let $E$ be a union of finitely many continua in $\overline\C$, let $F$ be a
compact set of positive capacity in $\overline\C\setminus E$, and let
$\{\Phi_n\}$ be holomorphic in a neighbourhood $\Omega$ of $F$ with
$\Omega\cap E=\emptyset$. Suppose
\begin{enumerate}
\item[$1^\circ$.] $\dfrac1{2n}\log\dfrac1{|\Phi_n(\zeta)|}\rightrightarrows
      \varphi(\zeta)$ on $\Omega$;
\item[$2^\circ$.] $(E,F,\varphi)$ has the $S$-property.
\end{enumerate}
Then, $F$ being a rectifiable arc,
\[
   \lim_{n\to\infty}\rho_n(f_n,E)^{1/n}=e^{-2w(E,F,\varphi)},
   \qquad
   f_n(z):=\int_F\frac{\Phi_n(t)}{t-z}\,dt ,
\]
where $\rho_n(\cdot,E)$ is the error of best approximation on $E$ by rational
functions of order at most $n$, that is of type $(n,n)$.
\end{theorem}

This is the form \cite[\S1.4]{GoncharRakhmanov} records for $F$ a rectifiable
arc, in which the density is integrated directly and the third hypothesis of
their Theorem 1 --- that the jump of $f$ across $F$ be non-vanishing --- is
carried by $\Phi_n$, which below is zero-free on $\Omega$. It is the form used
in their own \S2.

In \cite[\S2]{GoncharRakhmanov} this is applied with $E=[0,+\infty]$ and
$\Phi_n(t)=e^{-nt}$, for which $1^\circ$ gives the field
$\varphi(\zeta)=\tfrac12\operatorname{Re}\zeta$; the arc $F^{*}$ solving
$2^\circ$ is constructed there, and $e^{-2w}$ is identified with $H$. Their
$\gamma$ is $F^{*}$, joining $b=3+i\beta$ to $\bar b$ through
$\operatorname{Re}z<3$, together with two rays parallel to the real axis
joining $b,\bar b$ to $\infty$ in $\operatorname{Re}z>3$, oriented to wind once
about $E$. The number $3$ is theirs and is arbitrary. Their class of admissible
arcs is indexed by the abscissa, the equilibrium charge sits on a subarc with
endpoints $a,\bar a$ that the choice does not move, and the rest of $F^{*}$ may
be taken straight, with $\beta=\operatorname{Im}a$; all the abscissa has to do
is make the rays negligible. They fix it at $3$ because $e^{-3}<H$, noting that
``the choice of the number $3$ is connected only with this inequality''. It is
the same inequality that is used below.

Everything now follows from one observation: replacing $e^{-nt}$ by
$u_1(-nt)e^{-nt}$ changes $\Phi_n$ by a factor that grows subexponentially in
$n$, and $1^\circ$ divides by $2n$.

\begin{lemma}[Cauchy representation]\label{lem:cauchy}
Let $u_1\not\equiv0$ be meromorphic on $\C$ with finitely many poles and
$\log|u_1(z)|\le\varepsilon|z|$ for every $\varepsilon>0$ and all large $|z|$,
which is the upper bound of (b) below and all that is used here. Let $\Pi$ be
the sum of the principal parts
of $u_1\exp$ at the poles of $u_1$ --- a rational function, the poles being
finite in number --- and put $F:=u_1\exp-\Pi$, an entire function. Then for
every $n$ large enough that $\gamma$ winds once about
$-w/n$ for each pole $w$ of $u_1$,
\begin{equation}\label{eq:cauchyrep}
   F(-nu)=\frac1{2\pi i}\oint_\gamma u_1(-nt)\,e^{-nt}\,\frac{dt}{t-u},
   \qquad u\in E .
\end{equation}
\end{lemma}

\begin{proof}
Write $g:=u_1\exp$, meromorphic with poles exactly those of $u_1$. If $g$ has
principal part $P_w(z)=\sum_{j\le q}a_j(z-w)^{-j}$ at $w$, then
$\operatorname{Res}_{s=w}g(s)/(s-z)=-P_w(z)$, since the $j$-th term
contributes $a_j(-1)^{j-1}(w-z)^{-j}=-a_j(z-w)^{-j}$. Let $C$ wind once about
$z$ and about every pole of $u_1$; as $\operatorname{Re}s\to-\infty$ along $C$
the integrand decays, since $|g(s)|\le e^{\operatorname{Re}s+o(|s|)}$, so the
residue theorem applies and
\[
   \frac1{2\pi i}\oint_C\frac{g(s)}{s-z}\,ds=g(z)-\Pi(z)=F(z).
\]
Now put $z=-nu$ and substitute $s=-nt$: then $ds=-n\,dt$ and
$s-z=-n(t-u)$, so $ds/(s-z)=dt/(t-u)$ exactly, and $C=-n\gamma$ winds once
about $-nu$ and about every $w$ precisely when $\gamma$ winds once about $u$
and about every $-w/n$. Since $u\in E$ and $-w/n\to0\in E$, both hold for
large $n$.

The residue at $s=z$ is $g(z)$ only when $z$ is not itself a pole of $u_1$, so
the above proves \eqref{eq:cauchyrep} for every $u\in E$ other than the
finitely many points $-w/n$. It holds at those too, by continuation. Let $D$
be the region $\gamma$ winds once about: open, connected, containing every
finite point of $E$ and, for $n$ large, the points $-w/n$. The left side of
\eqref{eq:cauchyrep} is entire in $u$. The right side is holomorphic on $D$:
for $u\in D$ the factor $t-u$ does not vanish on $\gamma$, and the integral
converges absolutely and locally uniformly in $u$, since the integrand decays
like $e^{-n\operatorname{Re}t(1-\epsilon)}$ along the rays while $|t-u|$ stays
bounded below on compact subsets of $D$. Two functions holomorphic on $D$ that
agree off finitely many points agree on $D$. The point is not idle --- for
$u_1=z^{-\ell}$ the excluded point is $u=0$, an endpoint of $E$, and
Corollary~\ref{thm:conj} evaluates there.
\end{proof}

For $u_1(z)=z^{-\ell}$ the principal part of $e^z/z^\ell$ at the origin is
$s_\ell(z)/z^\ell$, so $F=\varphi_\ell$ and \eqref{eq:cauchyrep} is
Lemma~\ref{lem:rep}(iii), rescaled. Lemma~\ref{lem:cauchy} is nothing but that
identity with $z^{-\ell}$ replaced by an arbitrary such $u_1$, and it is
proved the same way; the $\varphi$-case is what suggested it.

\begin{theorem}\label{thm:class}
Let $u_0$ be rational and let $u_1\not\equiv0$ be meromorphic on $\C$ with
\begin{enumerate}
\item[(a)] finitely many poles, and
\item[(b)] $\log|u_1(z)|\le\varepsilon|z|$ for every $\varepsilon>0$ and all
      large $|z|$, with the matching lower bound
      $\log|u_1(z)|\ge-\varepsilon|z|$ holding uniformly on each sector
      $|\arg z-\pi|\ge\delta>0$.
\end{enumerate}
Suppose $f=u_0+u_1\exp$ has no pole on $\Rm$. Then
\[
   \lim_{n\to\infty}E_n(f)^{1/n}=H .
\]
In particular this holds for every rational $u_1\not\equiv0$.
\end{theorem}

The asymmetry in (b) is deliberate, and it is what admits zeros in infinite
number. A lower bound can fail only near a zero of $u_1$; the proof needs one
only on a set the argument keeps in a sector away from $\Rm$, and asks nothing
beyond the upper bound elsewhere, so the zeros are free to accumulate on $\Rm$.
For rational $u_1$, and whenever $u_1$ has finitely many zeros, the two halves
of (b) say the same thing. Infinitely many zeros are allowed if they crowd
towards $\Rm$: $u_1(z)=\sinh(\pi\sqrt z)/(\pi\sqrt z)$, with a zero at every
$-k^2$, satisfies (a) and (b), as Remark~\ref{rem:scope} checks.

\begin{proof}
Let $F=u_1\exp-\Pi$ as in Lemma~\ref{lem:cauchy}. Then $f-F=u_0+\Pi$ is
rational, of some degree $d$ independent of $n$. Since $F$ is entire and $f$
has no pole on $\Rm$, neither has $f-F$. Moreover $F$ is bounded on $\Rm$:
it is entire, and both $u_1e^x$, which is $e^{x+o(|x|)}$ by the upper bound in
(b), and $\Pi(x)$ tend to $0$ as $x\to-\infty$. Adding a fixed rational
function of degree $d$ shifts the index by at most $d$: if $R$ has type
$(n,n)$ then $R-(u_0+\Pi)$ has type $(n+d,n+d)$ and approximates $F$ exactly as
well as $R$ approximates $f$, while $S+(u_0+\Pi)$ has type $(n,n)$ whenever $S$
has type $(n-d,n-d)$, so
\[
   E_{n+d}(F)\ \le\ E_n(f)\ \le\ E_{n-d}(F),
\]
whence $\lim E_n(f)^{1/n}=\lim E_n(F)^{1/n}$ if either exists. Replacing
$x$ by $-u$ and using that $u\mapsto nu$ maps $E=[0,+\infty]$ onto itself and
preserves the order of a rational function,
\begin{equation}\label{eq:scal}
   E_n(F)=\rho_n\bigl(F(-\,\cdot\,),E\bigr)
        =\rho_n\bigl(F(-n\,\cdot\,),E\bigr),
\end{equation}
the device of \cite[(6)]{GoncharRakhmanov}. By Lemma~\ref{lem:cauchy},
$F(-nu)$ is the integral \eqref{eq:cauchyrep} over $\gamma$.

Split $\gamma=F^{*}\cup\text{rays}$. On the rays $\operatorname{Re}t\ge3$ and
$|\operatorname{Im}t|=|\beta|$, so $|t-u|\ge|\beta|$ for every $u\in E$,
uniformly, while $|t|\le\kappa\operatorname{Re}t$ there for a constant
$\kappa$ depending only on $\beta$. By the upper bound in (b) --- the rays run
towards $\Rm$ after reflection, and there nothing more is available, nor
needed --- $|u_1(-nt)e^{-nt}|\le e^{-n\operatorname{Re}t+\epsilon_nn|t|}
\le e^{-n\operatorname{Re}t(1-\kappa\epsilon_n)}$ with $\epsilon_n\to0$. Let
$h_n(u)$ be the integral over the rays. Its integrand is holomorphic in $u$
off the rays, and the bound just given is integrable along them uniformly for
$u$ in the strip $|\operatorname{Im}u|<|\beta|$, so $h_n$ is holomorphic on
that strip, a neighbourhood of $E$; on $E$ itself $|t-u|\ge|\beta|$ gives
$\|h_n\|_E=O\bigl(e^{-3n(1-\kappa\epsilon_n)}\bigr)$, and $h_n(u)\to0$ as
$u\to+\infty$. Since $e^{-3}<H$ this is
$o(H^n)$ for large $n$, and $h_n$ is one fixed function for each $n$, so
$|\rho_n(g_n+h_n,E)-\rho_n(g_n,E)|\le\|h_n\|_E=o(H^n)$, where
$g_n(u):=\int_{F^{*}}u_1(-nt)e^{-nt}(t-u)^{-1}dt$.

Apply Theorem~\ref{thm:GR} to $g_n$, that is with
$\Phi_n(t):=u_1(-nt)e^{-nt}$, $F=F^{*}$ and
$\varphi(\zeta)=\tfrac12\operatorname{Re}\zeta$. Take a bounded neighbourhood
$\Omega$ of $F^{*}$ with $\overline\Omega\cap E=\emptyset$, which exists
because $F^{*}$ is compact and disjoint from $E$. Three things are to be
checked.
\begin{enumerate}
\item[(i)] \emph{$\Phi_n$ is holomorphic and zero-free on $\Omega$ for
$n\ge n_0$.} On $\Omega$, $|t|$ is bounded above and away from $0$, so
$|{-nt}|\to\infty$ uniformly. Because $E=[0,+\infty]$, the set $-\Omega$ is
disjoint from $\Rm$, and hence so is $-\lambda\Omega$ for every $\lambda>0$:
the zeros of $u_1$ on $\Rm$ are never met, however many there are. The same
disjointness places $\Omega$ in a sector $|\arg t|\ge\delta'$, $\delta'>0$,
since $\overline\Omega$ is compact and misses $[0,+\infty]$, so $-n\Omega$
lies in $|\arg z-\pi|\ge\delta'$ for every $n$. The lower bound in (b),
uniform there, leaves $u_1$ no zeros of large modulus in that sector, and by
(a) the poles are finite in number; both are avoided once $n\delta$ is large.
Discarding finitely many $n$ changes no limit.
\item[(ii)] \emph{Condition $1^\circ$, with the field of $\exp$.} This is
where the lower bound in (b) is asked, and it is the only place. With
$|{-nt}|\to\infty$ uniformly on $\Omega$ and $-n\Omega$ inside the sector,
(b) gives $\log|u_1(-nt)|=o(n)$ uniformly on $\Omega$, so
\[
   \frac1{2n}\log\frac1{|\Phi_n(t)|}
   =\tfrac12\operatorname{Re}t-\frac{\log|u_1(-nt)|}{2n}
   \ \rightrightarrows\ \tfrac12\operatorname{Re}t
   \qquad\text{on }\Omega :
\]
the factor $u_1(-nt)$ grows subexponentially in $n$, and $1^\circ$ divides by
$2n$.
\item[(iii)] \emph{Condition $2^\circ$.} It involves $E$, $F$ and $\varphi$
alone, and $(E,F^{*},\tfrac12\operatorname{Re}\zeta)$ is the triple of
\cite[\S2]{GoncharRakhmanov}, which has the $S$-property there, with the same
$w$.
\end{enumerate}
Hence $\rho_n(g_n,E)^{1/n}\to e^{-2w}=H$. Since
$\|h_n\|_E/\rho_n(g_n,E)\to0$, also $\rho_n(g_n+h_n,E)^{1/n}\to H$; and
$F(-nu)=(2\pi i)^{-1}(g_n+h_n)(u)$, where the constant is harmless because
$(2\pi)^{-1/n}\to1$. With \eqref{eq:scal} and $E_{n+d}(F)\le E_n(f)\le
E_{n-d}(F)$, $E_n(f)^{1/n}\to H$.
\end{proof}

\begin{corollary}[Conjecture 2.1 of \cite{SchmelzerTrefethen}]\label{thm:conj}
For every $\ell\ge0$, $\ \lim_{n\to\infty}E_n(\varphi_\ell)^{1/n}=H$.
\end{corollary}

\begin{proof}
For $\ell=0$ this is the `1/9'-theorem of \cite{GoncharRakhmanov}. For
$\ell\ge1$, Lemma~\ref{lem:rep}(i) gives $\varphi_\ell=u_0+u_1\exp$ with
$u_1=z^{-\ell}$ and $u_0=-s_\ell/z^\ell$, both rational, and
Lemma~\ref{lem:rep}(ii) shows $\varphi_\ell$ has no pole on $\Rm$. Apply
Theorem~\ref{thm:class}; here $\Pi=s_\ell/z^\ell=-u_0$, so $F=\varphi_\ell$ and
$d=0$, and \eqref{eq:cauchyrep} is \eqref{eq:cauchyphi} rescaled.
\end{proof}

\begin{remark}[scope]\label{rem:scope}
For rational $u_1$, Theorem~\ref{thm:class} is Theorem 1 of
\cite{StahlSchmelzer}, whose proof was deferred to a manuscript that never
appeared. The two remaining hypotheses cannot be weakened: $f$ must have no
pole on $\Rm$, or $E_n(f)$ is infinite, and $u_1\not\equiv0$, or $f$ is
rational and $E_n(f)=0$ eventually. Nothing is asked of the \emph{location} of
the zeros and poles of $u_1$ off $\Rm$, and a pole of $u_1$ may sit on $\Rm$
provided $u_0$ cancels it, as it does for $\varphi_\ell$.

It is worth seeing what each half of (b) does. The lower bound is $1^\circ$
itself, and cannot be dropped:
$1^\circ$ asks for uniform convergence of $\log|\Phi_n|^{-1}$ on $\Omega$, and
that fails at any zero of $\Phi_n$ there. That it is asked only away from $\Rm$
buys more than it looks, because $-\lambda\Omega$ misses $\Rm$ for every
$\lambda>0$: the zeros of $u_1$ on $\Rm$ are never seen, however many there
are, and the theorem is genuinely wider than the rational case. It applies, for
instance, to
\[
   u_1(z)=\frac{\sinh\pi\sqrt z}{\pi\sqrt z}=\prod_{k\ge1}\Bigl(1+\frac
   z{k^2}\Bigr),
\]
entire of order $\tfrac12$, with infinitely many zeros, all of them at
$-k^2\in\Rm$. This is where the asymmetry of (b) earns its keep. The upper
bound holds everywhere, since $|\sinh\pi\sqrt z|\le e^{\pi|\sqrt z|}$; on
$|\arg z-\pi|\ge\delta$ one has $\operatorname{Re}\sqrt z\ge|\sqrt
z|\sin(\delta/2)$, so $\bigl|\log|u_1(z)|\bigr|=O(\sqrt{|z|})$ there and the
lower bound holds too. Asked on all of $\C$, the lower bound would fail at
every $z=-k^2$, and the example --- the one that separates the theorem from
the rational case --- would be excluded by its own hypothesis.

The lower bound also settles what is allowed off $\Rm$. Whether it holds is
a matter of growth, but what it demands of the zeros is only this: finitely
many in each sector away from $\Rm$. Finitely many off $\Rm$ altogether, the
simplest way to meet that, is not required; zeros off $\Rm$ may be infinite in
number provided they recede to $\Rm$ in argument, as $-k^2+ik$ do.

The widening comes from the zeros and not from the growth rate: if $u_1$ is meromorphic with finitely
many zeros \emph{and} finitely many poles, then $u_1=R_0e^{g}$ with $R_0$
rational and $g$ entire, so the upper bound in (b) --- which is all that is
needed here --- forces $\operatorname{Re}g(z)\le o(|z|)$; by
Borel--Carath\'eodory, which asks for nothing more, $\max_{|z|\le r/2}|g|=o(r)$,
and Cauchy's estimates on
$|z|=r/2$ then kill every coefficient of $g$ beyond the constant. So within that class (b)
already implies rationality, and one gains nothing by relaxing growth alone.

Nothing in the equilibrium analysis had to change: the field, the arc $F^{*}$,
the $S$-property, the elliptic parametrisation and the constant are those of
\cite[\S2]{GoncharRakhmanov}. This is worth contrasting with the route
projected in \cite[\S4]{StahlSchmelzer} for the missing proof. There the
analysis of \cite{GoncharRakhmanov} was to be redone from the inside --- the
Hermite--Walsh formula, the orthogonality of the denominator, the special
potential $p_0$ --- and the passage to the class was expected to be ``more
complicated'', carried out ``with the help of a special weighted norm''
$\|g\|_{u}=\|u_1^{-1}g\|_{\Rm}$, with respect to which the oscillatory
property of the best approximants and every subsequent step would have to be
re-established; it was noted as ``critical'' that $p_0$ survives the change.
That last observation is the same one made here, but it is used differently.
Once it is seen that $u_1$ may be absorbed into the weight $\Phi_n$ of
Theorem~\ref{thm:GR}, no weighted norm is needed and no step has to be
adapted, because \cite{GoncharRakhmanov} is applied rather than reproved. The
extension its authors expected to cost a paper costs a change of $\Phi_n$. What the present argument contributes is the
observation that the Cauchy representation, not the approximant, is the thing
to transfer --- Remark~\ref{prop:notransfer} says the approximant cannot
be --- and that $\varphi_\ell$ possesses such a representation. That is
the contour integral of Lemma~\ref{lem:rep}(iii), and it is the whole of the
input from outside \cite{GoncharRakhmanov}.
\end{remark}

\section{What the method does not give}\label{sec:further}

The proof of Theorem~\ref{thm:class} returns a number and no approximant, and
two questions that matter in practice lie outside its reach.

The first is whether one set of poles can serve every $\varphi_\ell$. A partial
fraction evaluation of $\varphi_\ell(hA)v$ costs one shifted solve
$(A-z_jI)x_j=v$ per pole, so if the $\varphi_\ell$ for $\ell=0,\dots,L$ could
share their poles the cost of a whole exponential integrator step would be that
of a single $\varphi$. This is the strategy of \cite[\S4]{SchmelzerTrefethen},
where the common-pole approximations generated there are found to be ``far from
optimal'', and of the shared-denominator Carath\'eodory--Fej\'er construction of
Al-Mohy \cite{SharedPole}. Nothing in Theorem~\ref{thm:class} bears on it: the
rate is a statement about each $\varphi_\ell$ separately. What is needed is an
explicit approximant, and there is one --- correcting the numerator of
$\rstar{m,m}(\exp)=P/S$ by $S\tau$, with $\tau$ the $\ell$-jet at the origin of
the classical error, keeps the denominator $S$ and makes the error exactly
$(\exp-P/S-\tau)/z^\ell$ --- but the analysis belongs with the numerical
question rather than here.

The second is whether the poles escape. Their limiting distribution is known,
but it is a measure-level statement and cannot exclude a stray pole near the
origin, which is exactly what would
spoil a partial fraction evaluation. Computation up to degree $20$ shows no
such pole, the distance from the poles to $\Rm$ growing linearly, but this is
the case $f=\varphi_\ell$ of Conjecture 1 of \cite{StahlSchmelzer} and remains
open.

\section*{Acknowledgements}

My thanks go to Nick Trefethen, who introduced me to Herbert Stahl and who was
the first to conjecture that the rate for the $\varphi$-functions is again
Halphen's constant --- the statement recorded as Conjecture 2.1 of
\cite{SchmelzerTrefethen} and proved here as Corollary~\ref{thm:conj}. Without
that guess, and without the collaboration with Stahl that it led to, this note
would not exist. I also thank the editor who saw that the growth condition of Theorem~\ref{thm:class}, as first
submitted, was two-sided and hence --- by the argument of
Remark~\ref{rem:scope} itself --- forced $u_1$ to be rational. Condition (b)
has its present form because of that observation.

\section*{Use of artificial intelligence}

Claude (Anthropic), run through Claude Code, was used at several stages of this
work, and the record of that use is the commit history of the repository from
which this paper is built. It was used as a critic while the argument was
taking shape: statements were put to it and its objections answered, which is
how the hypotheses of Theorem~\ref{thm:class} were weakened to their present
form and how a gap in an earlier version of the argument came to light. It was
used to write and run the scripts behind the figure and the numerical checks
reported here, all of which are included with the paper. And it was used
editorially, on the exposition and on the reference list. Every mathematical
statement was checked by me, and I am responsible for the mathematics and for
the paper as it stands.


\begin{thebibliography}{99}

\bibitem{SharedPole}
A.~H.~Al-Mohy, Shared-pole Carath\'eodory--Fej\'er approximations for linear
combinations of $\varphi$-functions, Mathematics, 13:24 (2025), art.~3985,
DOI 10.3390/math13243985.

\bibitem{Aptekarev}
A.~I.~Aptekarev, Sharp constants for rational approximations of analytic
functions, Sb. Math., 193:1 (2002), pp.~1--72,
DOI 10.1070/SM2002v193n01ABEH000619.

\bibitem{CRV}
A.~J.~Carpenter, A.~Ruttan, and R.~S.~Varga, Extended numerical computations on
the `1/9' conjecture in rational approximation theory, in Rational
Approximation and Interpolation, P.~R.~Graves-Morris, E.~B.~Saff, and
R.~S.~Varga, eds., Lecture Notes in Math. 1105, Springer, Berlin, 1984,
pp.~383--411, DOI 10.1007/BFb0072427.

\bibitem{CMV}
W.~J.~Cody, G.~Meinardus, and R.~S.~Varga, Chebyshev rational approximations to
$e^{-x}$ in $[0,+\infty)$ and applications to heat-conduction problems,
J. Approx. Theory, 2:1 (1969), pp.~50--65, DOI 10.1016/0021-9045(69)90030-6.

\bibitem{Gonchar1986}
A.~A.~Gonchar, Rational approximations of analytic functions (Russian), in
Proceedings of the International Congress of Mathematicians, Vol.~1 (Berkeley,
1986), Amer. Math. Soc., Providence, RI, 1987, pp.~739--748.

\bibitem{GoncharRakhmanov}
A.~A.~Gonchar and E.~A.~Rakhmanov, Equilibrium distributions and degree of
rational approximation of analytic functions, Math. USSR-Sb., 62:2 (1989),
pp.~305--348, DOI 10.1070/SM1989v062n02ABEH003242.

\bibitem{GutknechtTrefethen}
M.~H.~Gutknecht and L.~N.~Trefethen, Real polynomial Chebyshev approximation by
the Carath\'eodory--Fej\'er method, SIAM J. Numer. Anal., 19:2 (1982),
pp.~358--371, DOI 10.1137/0719022.

\bibitem{Halphen}
G.~H.~Halphen, Trait\'e des Fonctions Elliptiques et de Leurs Applications,
Premi\`ere Partie, Gauthier-Villars, Paris, 1886.

\bibitem{Magnus1986}
A.~P.~Magnus, CFGT determination of Varga's constant `1/9', Preprint B-1348,
Institut Math\'ematique, Universit\'e Catholique de Louvain, Louvain-la-Neuve,
1986.

\bibitem{Magnus1994}
A.~P.~Magnus, Asymptotics and super asymptotics of best rational approximation
error norms for the exponential function (the `1/9' problem) by the
Carath\'eodory--Fej\'er method, in Nonlinear Numerical Methods and Rational
Approximation II, A.~Cuyt, ed., Math. Appl. 296, Kluwer, Dordrecht, 1994,
pp.~173--185, DOI 10.1007/978-94-011-0970-3\_14.

\bibitem{MagnusMeinguet}
A.~P.~Magnus and J.~Meinguet, The elliptic functions and integrals of the `1/9'
problem, Numer. Algorithms, 24:1--2 (2000), pp.~117--139,
DOI 10.1023/A:1019141226189.

\bibitem{MeinardusVarga}
G.~Meinardus and R.~S.~Varga, Chebyshev rational approximations to certain
entire functions in $[0,+\infty)$, J. Approx. Theory, 3:3 (1970),
pp.~300--309, DOI 10.1016/0021-9045(70)90054-7.

\bibitem{SaffVarga}
E.~B.~Saff and R.~S.~Varga, Some open problems concerning polynomials and
rational functions, in Pad\'e and Rational Approximation, E.~B.~Saff and
R.~S.~Varga, eds., Academic Press, New York, 1977, pp.~483--488,
DOI 10.1016/B978-0-12-614150-4.50047-2.

\bibitem{SchmelzerTrefethen}
T.~Schmelzer and L.~N.~Trefethen, Evaluating matrix functions for exponential
integrators via Carath\'eodory--Fej\'er approximation and contour integrals,
Electron. Trans. Numer. Anal., 29 (2007/08), pp.~1--18.

\bibitem{Schoenhage}
A.~Sch\"onhage, Zur rationalen Approximierbarkeit von $e^{-x}$ \"uber
$[0,\infty)$, J. Approx. Theory, 7:4 (1973), pp.~395--398,
DOI 10.1016/0021-9045(73)90042-7.

\bibitem{Stahl1985}
H.~Stahl, The structure of extremal domains associated with an analytic
function, Complex Variables Theory Appl., 4:4 (1985), pp.~339--354,
DOI 10.1080/17476938508814119.

\bibitem{Stahl1986}
H.~Stahl, Orthogonal polynomials with complex-valued weight function, I, II,
Constr. Approx., 2:1 (1986), pp.~225--240 and pp.~241--251,
DOI 10.1007/BF01893429 and DOI 10.1007/BF01893430.

\bibitem{StahlSchmelzer}
H.~Stahl and T.~Schmelzer, An extension of the `1/9'-problem, J. Comput. Appl.
Math., 233:3 (2009), pp.~821--834, DOI 10.1016/j.cam.2009.02.084.

\bibitem{TrefethenGutknecht}
L.~N.~Trefethen and M.~H.~Gutknecht, The Carath\'eodory--Fej\'er method for real
rational approximation, SIAM J. Numer. Anal., 20:2 (1983), pp.~420--436,
DOI 10.1137/0720030.

\bibitem{TWS}
L.~N.~Trefethen, J.~A.~C.~Weideman, and T.~Schmelzer, Talbot quadratures and
rational approximations, BIT, 46:3 (2006), pp.~653--670,
DOI 10.1007/s10543-006-0077-9.

\end{thebibliography}
\end{document}